\documentclass{amsart}
\usepackage[margin=1in]{geometry}

\usepackage{shuffle}
\usepackage{mathrsfs}
\usepackage{bbm}
\usepackage{ifthen}
\usepackage{tikz}
\usetikzlibrary{matrix,arrows,calc}
\usepackage{mathdots}
\usepackage{array}
\usepackage{multirow}
\usepackage{cjhebrew}
\usepackage{hyperref}
\usepackage{float}
\usepackage{caption}
\usepackage{amscd}
\usepackage{tabularx}
\usepackage{extarrows}
\usepackage{ytableau}
\usepackage{amsmath}
\usepackage{amsthm}
\usepackage{cjhebrew}
\usepackage{amsfonts}
\usepackage{amssymb}
\usepackage{mathtools}
\usepackage{soul}
\usepackage{color}

\newcommand{\shin}{\text{\</s>}}

\newtheorem{theorem}{Theorem}[section]
\newtheorem{conjecture}[theorem]{Conjecture}
\newtheorem{proposition}[theorem]{Proposition}
\newtheorem{corollary}[theorem]{Corollary}
\newtheorem{lemma}[theorem]{Lemma}
\theoremstyle{definition}
\newtheorem{definition}[theorem]{Definition}
\newtheorem{example}[theorem]{Example}
\newtheorem{remark}[theorem]{Remark}

\def\unprotectedboldentry#1{\textcolor{Red}{\textbf{#1}}} 
\def\boldentry{\protect\unprotectedboldentry} 

\newcommand{\tikztableau}[2][scale=0.6,every node/.style={font=\small}]{
 \def\newtableau{#2} 
 \begin{array}{c} 
 \begin{tikzpicture}[#1] 
 \coordinate (x) at (-0.5,0.5); 
 \coordinate (y) at (-0.5,0.5); 
 \foreach \row in \newtableau { 
 \coordinate (x) at ($(x)-(0,1)$); 
 \coordinate (y) at (x); 
 \foreach \entry in \row { 
 \ifthenelse{\equal{\entry}{X}} 
 { 
 \node (y) at ($(y) + (1,0)$) {}; 
 \fill[color=lightgray] ($(y)-(0.5,0.5)$) rectangle +(1,1); 
 \draw[color=black] ($(y)-(0.5,0.5)$) rectangle +(1,1); 
 } 
 { 
 \ifthenelse{\equal{\entry}{\boldentry X}} 
 { 
 \node (y) at ($(y) + (1,0)$) {}; 
 \fill[color=darkgray] ($(y)-(0.5,0.5)$) rectangle +(1,1); 
 \draw ($(y)-(0.5,0.5)$) rectangle +(1,1); 
 } 
 { 
 \node (y) at ($(y) + (1,0)$) {\entry}; 
 \draw ($(y)-(0.5,0.5)$) rectangle +(1,1);} } 
 }} 
 \end{tikzpicture} 
 \end{array}} 
 
\newcommand{\tikztableausmall}[1]{\tikztableau[scale=0.45,every node/.style={font=\rm\small}]{#1}}

\newcommand{\N}{\mathbb{N}}
\newcommand{\Z}{\mathbb{Z}}

\newcommand{\lx}{\mathfrak{L}}
\newcommand{\clx}{\mathfrak{CL}}

\title{Lexical tableaux and quasisymmetric functions}

\author[J. M. Campbell]{John M. Campbell}\address{Department of Mathematics and Statistics, Dalhousie University, Halifax, NS}\email{jh241966@dal.ca}

\author[S. Daugherty]{Spencer Daugherty}\address{Department of Mathematics, University of Colorado Boulder, Boulder, CO}\email{spencer.daugherty@colorado.edu}

\usepackage[url=false, backend=bibtex]{biblatex}
\begin{document}

\maketitle

\begin{abstract} We introduce dual lexical functions, which we construct as an iterated strict-minimum restricted product on Frobenius characteristics of the trivial representation of the cyclic group induced to the symmetric group, generalizing Grinberg's construction of the dual immaculate functions. 
To model these combinatorially, we introduce lexical tableaux, a generalization of immaculate tableaux in which each row is lexicographically minimal among its cyclic shifts. For a fixed shape, standard lexical tableaux can be counted by a hook length formula with a q-analogue, while summing over all shapes with a fixed size and number of rows gives the (q-)Stirling numbers of the first kind. The dual lexical functions are then the generating functions of lexical tableaux and form a basis of the quasisymmetric functions. The dual lexical functions have a positive expansion into the fundamental basis, with coefficients counting another new family of tableaux defined by conditions on the major index values of their rows. We also introduce column-strict lexical tableaux and their respective quasisymmetric functions. We show that standard column-strict lexical tableaux are in bijection with multidimensional Kreweras-type lattice paths. 
\end{abstract}

\vspace{0.1in}


\vspace{0.1in}


\section{Introduction}

Tableaux and tableaux-like objects are of fundamental importance in the representation  
 theory of the symmetric group and many related areas of algebraic combinatorics. Young tableaux, in particular, are especially  
 significant in the study of and application of symmetric group representations. 
A family of Young-like tableaux that arose in the  
 construction of a noncommutative analogue of the Schur symmetric functions are the \emph{immaculate tableaux} of Berg, Bergeron, Saliola, Serrano, and Zabrocki \cite{BergBergeronSaliolaSerranoZabrocki2014}. Immaculate tableaux were first introduced as the defining combinatorial objects for the immaculate and dual immaculate functions \cite{BergBergeronSaliolaSerranoZabrocki2014}.  These functions form Schur-like bases of $NSym$ and $QSym$ respectively, and are now well studied, including research on the indecomposable modules for the dual 
 immaculate basis \cite{BergBergeronSaliolaSerranoZabrocki2015}, Pieri rules for dual immaculate functions and generalizations 
\cite{BergeronSanchezOrtegaZabrocki2016,Li2018,NieseSundaramvanWilligenburgWang2024}, multiplicative structures of the immaculate 
 basis \cite{BergBergeronSaliolaSerranoZabrocki2017,Li2018}, the expansion of dual immaculate functions into {Y}oung quasisymmetric 
 {S}chur functions \cite{AllenHallamMason2018}, noncommutative {Bell} polynomials \cite{NovelliThibonToumazet2018}, 
 a generalization of the dual immaculate basis to the polynomial algebra \cite{MasonSearles2021}, the immaculate inverse {K}ostka 
 matrix \cite{LoehrNiese2021}, and a generalization of dual immaculate functions using partially commutative 
 variables \cite{Daugherty2024}.
 
 In \cite{Grinberg_2017}, Grinberg constructs the dual immaculate functions as strict-minimum restricted products of complete homogeneous symmetric functions. Each such $h_m$ is the Frobenius characteristic of the trivial representation of $S_m$, as well as the content generating functions of weakly increasing words of length $m$. We consider the same construction, but replace each $h_m$ with a necklace character $N_m$, the Frobenius characteristic of the trivial representation of the cyclic group $C_m$ induced to $S_m$. These necklace characters are equal to the content generating functions of necklaces of length $m$. We call the resulting quasisymmetric functions the dual lexical functions. We introduce lexical tableaux, a generalization of immaculate tableaux where the word in each row must be minimal among its cyclic shifts, as the combinatorial model for these functions. While standard immaculate tableaux are in bijection with set partitions, standard lexical tableaux are in bijection with permutations (specifically, approached through their cycle decompositions). As such, standard immaculate tableaux of a fixed size and number of rows are counted by the Stirling numbers of the second kind, whereas standard lexical tableaux of a fixed size and number of rows are counted by the unsigned Stirling numbers of the first kind, ${n \brack k}$. We also introduce a column-strict version of lexical tableaux, which generalize shin tableaux \cite{CampbellFeldmanLightShuldinerXu2014}, also known as extended tableaux \cite{assaf2018kohnertpolynomials}.

We highlight three main results here. 
First, we have our algebraic characterization of the dual lexical functions using necklace characters and Grinberg's strict-minimum restricted product. This recovers the dual lexical functions as a natural cyclic analogue to the dual immaculate functions without appealing to tableaux. 

\setcounter{section}{4}
\setcounter{theorem}{1}
\begin{theorem}
For $\alpha \models n$,
    \[\lx^*_{\alpha} = N_{\alpha_1} \prec (N_{\alpha_2} \prec ( \cdots \prec (N_{\alpha_k} \prec 1) \cdots))  \qquad \text{ where } \qquad N_m=\mathrm{ch}\big(\mathrm{Ind}_{C_m}^{S_m}\mathbf 1\big).\]
\end{theorem}

Then, we find a positive expansion of the dual lexical functions into the fundamental basis in $QSym$, along with the dual expansion of the ribbon functions into the lexical basis in $NSym$. Unlike in the immaculate case, the types of lexical tableaux that standardize to a fixed standard lexical tableau need not form an interval in refinement order, so standardization of lexical tableaux does not yield a fundamental expansion. Instead, we need another type of tableaux, which we call maj-divisible tableaux, that have a more compatible standardization and descent structure. These are tableaux with the condition that their row minima must increase in successive rows and the major index of each row word must be divisible by the length of the row. Lexical tableaux and maj-divisible tableaux have a correspondence following from a correspondence between necklaces and maj-divisible words that appears in \cite{ahlbach2018cyclicsievingnecklacesbranching} using a cyclic sieving argument. Namely, necklaces and maj-divisible words have the same weight generating functions, and thus so do lexical tableaux and maj-divisible tableaux. Unlike lexical tableaux, maj-divisible tableaux have nice standardization and descent structures, and so they produce a clean fundamental expansion.

\setcounter{section}{4}
\setcounter{theorem}{9}
\begin{theorem}
 We have the dual positive expansions 
 \[R_{\beta} = \sum_{\alpha \models n} J^{\lx}_{\alpha, \beta} \lx_{\alpha} \qquad \text{and } \qquad \lx^*_{\alpha}= \sum_{\beta \models n} J^{\lx}_{\alpha,\beta}F_{\beta},\] where $J^{\lx}_{\alpha, \beta}$ counts the number of standard maj-divisible tableaux of shape $\alpha$ with descent composition $\beta$.
\end{theorem}
We also conjecture, based on computational evidence, that the dual lexical functions expand positively into the Young quasisymmetric Schur basis and the dual column-strict lexical functions expand positively into the fundamental basis.

Finally, we have an interesting correspondence between multidimensional Kreweras-type lattice paths and column-strict lexical tableaux. For rectangle shapes, the lattice paths will all be excursions (paths that start and end at the origin) which are of special interest \cite{asinowski2023kreweras}.

\setcounter{section}{5}
\setcounter{theorem}{3}
\begin{theorem}
    For a partition $\lambda = (\lambda_1, \ldots, \lambda_{\ell})\vdash n$, standard column-strict lexical tableaux of shape $\lambda$ are in bijection with lattice paths of length $n$ that start at the origin and end at $(\lambda_1'-\lambda_2', \lambda_1'-\lambda_3', \ldots, \lambda_1'-\lambda_{\lambda_1}')$ while staying in the $\lambda_1-1$ dimensional nonnegative orthant, with steps $\mathbf{1}$, $-\mathbf{e}_1$, $-\mathbf{e}_2$, $\ldots$, and $-\mathbf{e}_{\lambda_1-1}$.
\end{theorem}

The paper is structured as follows. In Section \ref{sectionLexical}, we introduce lexical tableaux. First, we study their enumeration and provide a hook length formula for the number of standard lexical tableaux. We also present a $q$-hook length formula that refines the standard lexical tableaux count using inversions on reading words. We also formalize the relationship between standard lexical tableaux and permutations, including a $q$-analogue count that yields the $q$-Stirling numbers of the first kind.
Next, we briefly discuss standardization and descents in standard lexical tableaux. We follow this by introducing maj-divisible tableaux and studying their standardization and descent structure.

In Section \ref{sectionqsym}, we define dual lexical quasisymmetric functions in terms of lexical tableaux and present the equivalent algebraic construction of the dual lexical functions. Then, we prove their positive expansions into the monomial and fundamental bases. The coefficients in the fundamental expansion count standard maj-divisible tableaux with certain descent sets. We also conjecture that the dual lexical functions expand positively into the Young quasisymmetric Schur basis. We then define the lexical functions in $NSym$ as the basis dual to the dual lexical functions.  We give a positive expansion of elementary noncommutative symmetric functions into the lexical basis.
We close this section with some results on the antipode applied to lexical functions, including an explicit antipode formula for two-row shapes, and discussion on product formulas involving lexical functions.

In Section \ref{sectioncolumnstrict}, we introduce column-strict lexical tableaux, which impose an additional column-strictness condition on lexical tableaux. These tableaux generalize shin tableaux \cite{CampbellFeldmanLightShuldinerXu2014} (also known as extended tableaux \cite{assaf2018kohnertpolynomials}) in the way that lexical tableaux generalize immaculate tableaux. Adding the column-strictness condition also takes full advantage of the tableaux format. We find standard column-strict lexical tableaux to be in bijection with Kreweras-type lattice paths. Specifically, rectangular standard column-strict lexical tableaux are in bijection with (higher dimensional) Kreweras excursions, a count for which is not currently known for $m>3$ \cite{Bostan_2016}. In Section \ref{sectionCSLTqsym}, we define column-strict lexical and dual column-strict lexical functions using column-strict lexical tableaux. These bases relate closely to the shin functions \cite{CampbellFeldmanLightShuldinerXu2014} and the extended Schur functions \cite{assaf2018kohnertpolynomials}. We examine expansions of dual column-strict lexical functions into various bases and conjecture that they have a positive expansion into the fundamental basis.

\setcounter{section}{1}
\section{Preliminaries}\label{sectionPreliminaries}

 We highlight Macdonald's text \cite{Macdonald1995} as the usual monograph on symmetric functions. A review of preliminaries on 
 symmetric functions, as below, is required for our purposes. 

  An \emph{integer partition} is a weakly decreasing and finite tuple $\lambda = (\lambda_1, \lambda_2, \ldots, \lambda_{\ell(\lambda)})$ 
 of positive integers, where $|\lambda|$ denotes the sum of the parts of $\lambda$, and we write $\lambda \vdash n$ if $|\lambda| = 
 n$. Similarly, an \emph{integer composition} $\alpha = (\alpha_1, \alpha_2, \ldots, \alpha_{\ell(\alpha)})$ is a finite tuple of positive 
 integers, where we write $|\alpha|$ in place of the sum of the entries of $\alpha$, and we write $\alpha \vDash n$ if $|\alpha| = n$.
 A  
 \emph{tableau}, for the purposes of this paper,   may be understood as a two-dimensional arrangement of (labeled or unlabelled) cells 
 that are positioned into left-justified rows,   with the \emph{shape} of a tableau $T$ forming an integer composition $\alpha$ such that  
 the number of cells in the $i^{\text{th}}$   row of $T$ (from the bottom according to the so-called French convention) is equal to  
 $\alpha_i$ for $i \in \{ 1, 2, \ldots, \ell(\alpha) \}$. 

 The rings and algebras considered in this paper will be over $\mathbb{Q}$ for convenience and by convention. 
Let $\mathrm{Sym}_k$ denote the space of formal power series in
$x_1,x_2,\dots$ that are homogeneous of degree $k$ and invariant under every
permutation of the variables. Then, the algebra of symmetric functions is defined
\begin{equation}
\mathrm{Sym}:=\bigoplus_{k\ge 0}\mathrm{Sym}_k .
\end{equation}

 For a composition $\alpha$, let $sort(\alpha)$ denote the partition
obtained by sorting the parts of $\alpha$ into weakly decreasing order. For a given integer partition $\lambda$, we set
\begin{equation*}
 m_{\lambda} = \sum_{sort(\alpha)=\lambda,} \sum_{ i_{1} < i_{2} < \cdots < i_{\ell(\lambda)} } 
 x_{i_{1}}^{\alpha_{1}} x_{i_{2}}^{\alpha_{2}} \cdots 
 x_{i_{\ell(\lambda)}}^{\alpha_{\ell(\lambda)}}.
\end{equation*}
 Letting $\mathcal{P}$ denote the set of all integer partitions, we may thus define the \emph{monomial basis} of $\textsf{Sym}$ as 
 $\{ m_{\lambda} \}_{\lambda \in \mathcal{P}}$.  By then setting $h_0 = 1$ and $ h_r = 
 \sum_{\lambda \vdash r} m_{\lambda} $ for $r > 0$, we then set $h_{\lambda} = h_{\lambda_{1}} h_{\lambda_{2}} \cdots 
 h_{\lambda_{\ell(\lambda)}}$ for an integer partition $\lambda$, giving rise to the \emph{complete homogeneous basis} $\{ 
 h_{\lambda} \}_{\lambda \in \mathcal{P}}$ of $\textsf{Sym}$. 

 A \emph{semistandard Young tableau} is a tableau of a given partition shape $\lambda$ with each cell labeled with a positive integer 
 and with weakly increasing rows and strictly increasing columns. The \emph{content} of a SSYT $T$ is the finite tuple $t = (t_{1}, t_{2}, 
 \ldots, t_{\ell(t)})$ such that $\ell(t)$ is the maximal label in $T$ and such that the number of labels equal to $i$ is $t_i$ for $i \in \{ 1, 2, 
 \ldots, \ell(t) \}$. We let the \emph{Kostka coefficient} $K_{\lambda, \mu}$ be defined as the number of SSYTs of shape $\lambda$ 
 and content $\mu$. For the case whereby $\mu = (1^{|\lambda|})$, a SSYT of shape $\lambda$ and content $\mu$ is said to be 
 \emph{standard}, and the number of standard Young tableaux of shape $\lambda$ is denoted as $f^{\lambda}$.  This allows us to define the \emph{Schur basis} $\{ s_{\lambda} \}_{\lambda \in \mathcal{P}}$ so that $ h_{\mu} = \sum_{\lambda} 
 K_{\lambda, \mu} s_{\lambda}$.

\begin{example}\label{exampleSY}
 For $\lambda = (2, 2, 1)$, we find that there are $f^{\lambda} = 5$
 standard Young tableaux of the specified shape, as below: 
 $$ \ytableausetup {centertableaux,boxsize=normal,notabloids}
 \begin{ytableau} 3 \\ 2 & 5 \\ 1 & 4 \end{ytableau} \ \ \ \ \ \ \ 
 \ytableausetup {centertableaux,boxsize=normal,notabloids}
 \begin{ytableau} 4 \\ 2 & 5 \\ 1 & 3 \end{ytableau} \ \ \ \ \ \ \ \ytableausetup {centertableaux,boxsize=normal,notabloids}
 \begin{ytableau} 4 \\ 3 & 5 \\ 1 & 2 \end{ytableau} \ \ \ \ \ \ \ \ytableausetup {centertableaux,boxsize=normal,notabloids}
 \begin{ytableau} 5 \\ 2 & 4 \\ 1 & 3 \end{ytableau} \ \ \ \ \ \ \ \ytableausetup {centertableaux,boxsize=normal,notabloids}
 \begin{ytableau} 5 \\ 3 & 4 \\ 1 & 2 \end{ytableau}. $$ 
\end{example}
 

 For an integer composition $\alpha$, we write
\begin{equation}\label{quasimonomial}
 M_{\alpha} = \sum_{i_{1} < i_{2} < \cdots < i_{\ell(\alpha)}}
 x_{i_{1}}^{\alpha_{1}} 
 x_{i_{2}}^{\alpha_{2}} \cdots 
x_{i_{\ell(\alpha)}}^{\alpha_{\ell(\alpha)}}. 
\end{equation}
 By then setting $\textsf{QSym}_{k} := \mathscr{L}\{ M_{\alpha} : \alpha \vDash k \}$, we then form a graded algebra by setting 
\begin{equation}\label{gradedQSym}
 \textsf{QSym} :=
 \bigoplus_{k=0}^{\infty} 
 \textsf{QSym}_{k}.
\end{equation}
 The algebra in \eqref{gradedQSym} is referred to as the \emph{algebra of quasisymmetric functions}, and elements in the basis $\{ 
 M_{\alpha} \}_{\alpha \in \mathcal{C}}$ of $\textsf{QSym}$ are referred to as \emph{monomial quasisymmetric functions}. For 
 compositions $\alpha$ and $\beta$, we write $\alpha \succeq \beta$ if $\alpha$ can be obtained by adding together consecutive parts 
 of $\beta$. This allows us to define the \emph{fundamental basis} $\{ F_{\alpha} \}_{\alpha \in \mathcal{C}}$ of $\textsf{QSym}$ so that 
 $ F_{\alpha} = \sum_{\alpha \succeq \beta} M_{\beta}$. The power sum analogues in $QSym$, introduced in \cite{BallantineDaughertyHicksMasonNiese2020}, are denoted $\Psi_{\alpha}$ and $\Phi_{\alpha}$. 
 
 Quasisymmetric functions were introduced by 
 Gessel in 1984 \cite{Gessel1984} and provide deep and major areas of study within algebraic combinatorics. An equivalent version of the 
 algebra $\textsf{NSym}$ dual to $\textsf{QSym}$ was introduced by Gelfand et al.\ in 1995 
 \cite{GelfandKrobLascouxLeclercRetakhThibon1995}, and the importance of $\textsf{NSym}$ within algebraic combinatorics is much like 
 that of its dual $\textsf{QSym}$. Setting $\textsf{NSym}_{k} := \mathscr{L}\{ H_{\alpha} : \alpha \vDash k \}$, letting $H_{\alpha}$ be seen 
 as a variable, we form the graded algebra $$ \textsf{NSym} := \bigoplus_{k=0}^{\infty} \textsf{NSym}_k $$ endowed with the 
 multiplicative operation such that $H_{\alpha} H_{\beta} = H_{\alpha \cdot \beta}$ for the concatenation $\alpha \cdot \beta$ of 
 $\alpha$ and $\beta$. The elementary basis $\{ E_{\alpha} \}_{\alpha \in \mathcal{C}}$ of $\textsf{NSym}$ may then be defined 
 according to the recursion $ E_{n} = \sum_{i=1}^{n} (-1)^{i-1} H_{i} E_{n-i}$, with $E_{\alpha} = E_{\alpha_1} E_{\alpha_2} \cdots 
 E_{\alpha_{\ell(\alpha)}}$, and the ribbon basis $ \{ R_{\alpha} \}_{\alpha \in \mathcal{C}} $ may be defined so that $ R_{\alpha} 
 = \sum_{\beta \succeq \alpha} (-1)^{\ell(\alpha) - \ell(\beta)} H_{\beta}$. The power sum analogues in $NSym$, introduced in \cite{GelfandKrobLascouxLeclercRetakhThibon1995}, are denoted $\mathbf{\Psi}_{\alpha}$ and $\mathbf{\Phi}_{\alpha}$.

 The duality between $\textsf{NSym}$ and $\textsf{QSym}$ may be demonstrated using the bases defined above, according to the 
 bilinear pairing $\langle \cdot, \cdot \rangle\colon \textsf{NSym} \times \textsf{QSym} \to \mathbb{Q}$ such that $\langle H_{\alpha}, 
 M_{\beta} \rangle = \delta_{\alpha, \beta}$ for the Kronecker delta function $\delta_{\cdot, \cdot}$, or, equivalently, such that $\langle 
 R_{\alpha}, F_{\beta} \rangle = \delta_{\alpha, \beta}$.  For general background material on Hopf algebras, we highlight standard references on Hopf algebras 
 \cite{DascalescuNastasescuRaianu2001,MilnorMoore1965}, and, for the Combinatorial Hopf Algebra
 structures on $\textsf{Sym}$, $\textsf{QSym}$, and $\textsf{NSym}$, we refer to the seminal paper by Aguiar et al.\ on CHAs \cite{AguiarBergeronSottile2006} and related references. 

 The \emph{immaculate basis} of $\textsf{NSym}$ may be defined by a Pieri rule, with 
\begin{equation}\label{immaculatePieri}
 \mathfrak{S}_{\alpha} H_{s} = \sum_{\beta} \mathfrak{S}_{\beta}, 
\end{equation}
 where the sum in \eqref{immaculatePieri} is over all compositions $\beta \vDash |\alpha| + s$
 that differ from $\alpha$ by an immaculate horizontal strip, i.e., so that 
 $\alpha_{j} \leq \beta_{j}$
 for all $j \in \{ 1, 2, \ldots, \ell(\alpha) \}$, and $\ell(\beta) \leq \ell(\alpha) + 1$ \cite{BergBergeronSaliolaSerranoZabrocki2014}. This basis also has a combinatorial interpretation in terms of  tableaux. An \emph{immaculate tableau} of shape $\alpha$ and content $\beta$ is a tableau $T$ of the specified shape such that the number of 
 labels in $T$ equal to $i$ is $\beta_i$ for $i \in \{ 1, 2, \ldots, \ell(\beta) \}$ and such that the first column of $T$ is strictly increasing and 
 such that the rows of $T$ are weakly increasing. Let $K_{\alpha, \beta}^{\mathfrak{S}}$ denote the number immaculate tableaux of 
 shape $\alpha$ and content $\beta$. For the case whereby $\beta = (1^{|\alpha|})$, an immaculate tableau of shape $\alpha$ and content 
 $\beta$ is said to be \emph{standard}, and we let $g^{\alpha}$ denote the number of standard immaculate tableaux of shape $\alpha$. 

\begin{example}
 In contrast to Example \ref{exampleSY}, for $\alpha = (1,2,2)$, we find that there are $g^{\alpha} = 3$ standard immaculate tableaux of 
 the given shape, as below: $$ \ytableausetup {centertableaux,boxsize=normal,notabloids}
 \begin{ytableau} 3 & 5 \\ 2 & 4 \\ 1 \end{ytableau} \ \ \ \ \ \ \ 
 \ytableausetup {centertableaux,boxsize=normal,notabloids}
 \begin{ytableau} 3 & 4 \\
 2 & 5 \\ 1 \end{ytableau}
 \ \ \ \ \ \ \ 
 \ytableausetup {centertableaux,boxsize=normal,notabloids}
 \begin{ytableau} 4 & 5 \\ 
 2 & 3 \\ 1 \end{ytableau}. $$ 
\end{example}

 For compositions $\alpha$ and $\beta$, 
 the \emph{lexicographic order} $\leq_{\ell}$ 
 may be defined recursively so that 
 $\alpha \geq_{\ell} \beta$ if
 $\alpha_1 > \beta_1$, or $\alpha_1 = \beta_1$
 and $(\alpha_2, \ldots, \alpha_{\ell(\alpha)}) \geq_{\ell} 
 (\beta_{2}, \ldots, \beta_{\ell(\beta)})$, 
 and similarly for words (over a totally ordered set). 
The iterative application of the Pieri rule in 
 \eqref{immaculatePieri} yields 
\begin{equation}\label{Htoimmaculate} 
 H_{\beta} = \sum_{\alpha \geq_{\ell} \beta} 
 K_{\alpha, \beta}^{\mathfrak{S}} \mathfrak{S}_{\alpha},
\end{equation} 
 for the number $K_{\alpha, \beta}$ of immaculate tableaux of shape $\alpha$ and content $\beta$, i.e., the number of tableaux of shape 
 $\alpha$ with a strictly increasing first column and with weakly increasing rows and with content $\beta$, i.e., so that the number of cells 
 with label $i$ is equal to $\beta_i$ for $i \in \{ 1, 2, \ldots, \ell(\beta) \}$. The expansion in \eqref{Htoimmaculate} provides a key to our 
 construction of an analogue of immaculate functions based on our lexical generalization of immaculate tableaux. 

The dual immaculate functions are defined as the dual basis to the immaculate functions and are given by 
\[\mathfrak{S}^*_{\alpha} = \sum_{T \in IT_{\alpha}}x^T = \sum_{\beta \leq_{\ell} \alpha} K^{\mathfrak{S}}_{\alpha, \beta}M_{\beta}.\] In \cite{Grinberg_2017}, Grinberg offers a construction of the dual immaculate function using a restricted product.
 The product is defined for monomials  $\mathfrak{m}$ and $\mathfrak{n}$ and extended linearly: 
\[\mathfrak{m} \prec \mathfrak{n} = \begin{cases}
    \mathfrak{m} \cdot \mathfrak{n} & \text{min(Supp $\mathfrak{m}$)} < \text{min(Supp $\mathfrak{n}$);}\\
    0 & \text{otherwise,}
\end{cases}\] where Supp $\mathfrak{m} = \{ i : a_i >0\}$ if $\mathfrak{m} = x_1^{a_1}x_2^{a_2}x_3^{a_3} \cdots$. Recall that $h_n$ is a complete homogeneous symmetric function. For any $\alpha,$
\[\mathfrak{S}^*_{\alpha} = h_{\alpha_1} \prec ( h_{\alpha_2} \prec ( \cdots \prec (h_{\alpha_{\ell}} \prec 1) \cdots )).\]
Grinberg also shows an alternate presentation of $h_m \prec$ as creation operators with a more classical form.

The immaculate and dual immaculate functions are known as Schur-like bases of $NSym$ and $QSym$, respectively. Other Schur-like bases include the shin and extended Schur functions \cite{assaf2018kohnertpolynomials, CampbellFeldmanLightShuldinerXu2014}, denoted $\shin_{\alpha}$ and $\shin^*_{\alpha}$,  as well as the Young noncommutative Schur and Young quasisymmetric Schur functions \cite{LuotoMykytiukWilligenburg2013}, denoted $\hat{\mathbf{s}}_{\alpha}$ and $\hat{S}_{\alpha}$.

\section{Lexical tableaux}\label{sectionLexical}
 Letting $ \left\{ \begin{smallmatrix} n \\ k \end{smallmatrix} \right\} $ denote the Stirling number of the second kind, which is the number of ways of partitioning the set $\{ 1, 2, \ldots, n \}$ into $k$ blocks, we obtain the following analogue of the well-known sum of counts of standard Young tableaux which equals $n!$. 

\begin{proposition}\label{thm:snk}
 For $1 \leq k \leq n$, we have 
\begin{equation}\label{displaysnk}
 \sum_{\substack{\alpha \vDash n\\ \ell(\alpha) = k}} g^{\alpha} = 
 \left\{ \begin{matrix} 
 n \\ k 
 \end{matrix} \right\}. 
\end{equation} 
\end{proposition}

\begin{proof} 
 We construct a bijection $f$ between standard immaculate tableaux of size $n$ with $k$ rows to set partitions of $\{ 1, 2, \ldots, 
 n \}$ with $k$ blocks as follows. Given a standard immaculate tableau $T$, we let $f(T)$ be the set partition $B_1 / B_2 / \cdots / 
 B_k$ where $i \in B_j$ if $i$ is in Row $j$ of $T$. Note that the order of the blocks is irrelevant since we are working with unordered set 
 partitions. The inverse map $f^{-1}$ takes a set partition $\pi = B_1 / B_2 / \cdots / B_k$ and constructs a standard immaculate 
 tableaux $f^{-1}(\pi)$ of shape $(|B_{j_1}|, |B_{j_2}|, \ldots, |B_{j_k}|)$ where $\min(B_{j_1}) < \min(B_{j_2}) < \cdots < 
 \min(B_{j_k})$, and where the entries of row $i$ are exactly the entries of $B_{j_i}$ sorted into increasing order. 
\end{proof}

Proposition \ref{thm:snk} illustrates how immaculate tableaux are useful and natural combinatorial objects, and the
 right-hand evaluation in \eqref{displaysnk} leads us to consider what would be appropriate as an analogue of 
 \eqref{displaysnk} based on immaculate-like tableaux and variants or generalizations of $ \left\{ \begin{smallmatrix} 
 n \\ k \end{smallmatrix} \right\}$. In this direction, the unsigned Stirling number $|s(n,k)|= {n \brack k}$ of the first kind is equal to the number of permutations of $n$ elements with $k$ disjoint cycles. Given a permutation $\sigma$ of $\{1,2,\ldots, n\}$ with $k$ disjoint cycles, we denote its cycle decomposition $c_{\sigma, 1}, c_{\sigma, 2}, \ldots, c_{\sigma, k}$, with each cycle written as a tuple of elements in $\{1,2,\ldots, n\}$. We order these tuples from least to greatest based on their minimal element.

 To construct a new family of tableaux based on an analogue of Proposition \ref{thm:snk} with 
 ${n \brack k} $ in place of 
 $ \left\{ \begin{smallmatrix} 
 n \\ k 
 \end{smallmatrix} \right\}$, 
 we employ the concept of a \emph{cyclic shift}, as seen in the work of Adin et al.\ \cite{AdinGesselReinerRoichman2021}. Given a word $w = w_1 w_2 \cdots w_{\ell(w)}$, 
 a cyclic shift of $w$ is given by $w^{(i)} = w_{i+1} w_{i+2} \cdots
 w_{\ell(w)} w_1 w_2 \cdots w_{i}$ for any $i \in [ \ell(w) 
 ] = \{ 1, 2, \ldots, \ell(w) 
 \}$. Let $[\vec{w}]$ denote the set of cyclic shifts of $w$. Define a \emph{necklace word} to be a word that is lexicographically minimal among all of its cyclic shifts.

\begin{definition}
 We define a \emph{lexical tableau} of shape $\alpha \vDash n$ and type (or content) $\beta$ as a filling of the diagram
 of $\alpha$ such that $i$ appears exactly $\beta_i$ times, the entries in the first column are strictly increasing, and the word
 $w_i$ formed by the entries of row $i$ (in order) is a necklace word.
\end{definition}

 We refer to lexical tableaux of type $\beta = (1)^n$ as \emph{standard}. Standard lexical tableaux are in bijection with permutations, where rows in tableaux correspond to cycles in cycle decompositions. 

\begin{proposition}\label{theoremltstirling}
 Let $lt_{\alpha}$ be the number of standard lexical tableaux of shape $\alpha$. Then 
\begin{equation}\label{displayltstirling}
 \sum_{\substack{ \alpha \vDash n,\\ \ell(\alpha)=k}} lt_{\alpha} = {n \brack k}. 
\end{equation} 
\end{proposition}

\begin{proof} 
 Let $T$ be a standard lexical tableau of size $n$ with $k$ rows. We map this to a permutation $\sigma$ of $[n]$ by setting $c_{\sigma, i}$ to be the tuple formed by the entries in row $i$ of $T$. This map is a bijection, 
 where the inverse map takes a permutation with cycle decomposition $c_{\sigma,1}c_{\sigma, 2} \cdots c_{\sigma, k}$ to a standard lexical tableau where the entries in row $i$ are given by the entries in $c_{\sigma, i}$, in the same order.
\end{proof}

\begin{example}\label{ex:tabs}
 For the ${4\brack2}=11$ case,   the standard lexical tableaux of size $4$ with $2$ rows are as below. \vspace{2mm}
\[
 \begin{ytableau}
 2 & 3 & 4\\
 1
 \end{ytableau} \qquad
 \begin{ytableau}
 2 & 4 & 3\\
 1
 \end{ytableau} \qquad
 \begin{ytableau}
 2\\
 1 & 3 & 4
 \end{ytableau} \qquad
 \begin{ytableau}
 2\\
 1 & 4 & 3
 \end{ytableau} \qquad
 \begin{ytableau}
 3\\
 1 & 2 & 4
 \end{ytableau} \qquad
 \begin{ytableau}
 3\\
 1 & 4 & 2
 \end{ytableau} \] \vspace{-4mm}
 
 \[
 \begin{ytableau}
 4\\
 1 & 2 & 3
 \end{ytableau} \qquad
 \begin{ytableau}
 4\\
 1 & 3 & 2
 \end{ytableau}\qquad
 \begin{ytableau}
 3 & 4\\
 1 & 2 
 \end{ytableau} \qquad
 \begin{ytableau}
 2 & 4\\
 1 & 3 
 \end{ytableau} \qquad
 \begin{ytableau}
 2 & 3\\
 1 & 4
 \end{ytableau} 
 \] \vspace{2mm}
 
 This illustrates how lexical tableaux generalize standard immaculate tableaux, since for the $ \left\{ \begin{smallmatrix} 4 \\ 2 \end{smallmatrix} \right\} = 7$ case, 
 the standard immaculate tableaux of size $4$ with $2$ rows are as below. \vspace{2mm}
\[
 \begin{ytableau}
 2 & 3 & 4\\
 1
 \end{ytableau} \qquad
 \begin{ytableau}
 2\\
 1 & 3 & 4
 \end{ytableau} \qquad
 \begin{ytableau}
 3\\
 1 & 2 & 4
 \end{ytableau} \qquad
 \begin{ytableau}
 4\\
 1 & 2 & 3
 \end{ytableau} \qquad
 \begin{ytableau}
 3 & 4\\
 1 & 2 
 \end{ytableau} \qquad
 \begin{ytableau}
 2 & 4\\
 1 & 3 
 \end{ytableau} \qquad
 \begin{ytableau}
 2 & 3\\
 1 & 4
 \end{ytableau} 
 \] \vspace{2mm}
\end{example}

\subsection{On the enumeration of lexical tableaux}\label{subsectionEnumeration}
 Consider a multiset $\mathcal{B} =\{a_1^{n_1}, a_2^{n_2}, \ldots, a_k^{n_k} \}$. The number of necklace words with characters corresponding exactly to the elements in $\mathcal{B}$ is given by \[N(\mathcal{B}) = \frac{1}{|\mathcal{B}|} \sum_{d|\gcd(n_1, \ldots, n_k)} \binom{|\mathcal{B}|/d}{n_1/d, \ldots, n_k/d} \varphi(d),\]
where $\varphi$ is Euler's totient function \cite{gilbert1961symmetry}. Let $IT_{\alpha,\beta}$ be the set of immaculate tableaux of shape $\alpha$ and type $\beta$. Let $\mathcal{R}^T_i$ be the multiset of entries in row $i$ of an immaculate tableau $T$. 
 Then, we have \[ K^{\lx}_{\alpha, \beta} = \sum_{T \in IT_{\alpha, \beta}} N(\mathcal{R}^T_1)N(\mathcal{R}^T_2) \cdots N(\mathcal{R}^T_{\ell(\alpha)}).\]

 The following result is key in relation to the 
 unitriangularity of transition matrices we later require. 

\begin{proposition}\label{theorem:lexical_coeff}
 Given a composition $\alpha \vDash n$, we have $K^{\lx}_{\alpha, \alpha}=1$ and, if $\alpha <_{\ell} \beta$, then $K^{\lx}_{\alpha, \beta}=0$.
\end{proposition}

\begin{proof}
 Consider an empty diagram of shape $\alpha = (\alpha_1, \ldots, \alpha_k)$ that we want to fill as a lexical tableau of type $\alpha$. If 
 there is a $1$ anywhere in a given row, then the first entry of that row must be a $1$. Since the first column is strictly increasing, there
 must be a $1$ in the first row, and there must not be any $1$s in the following rows. So all $\alpha_1$ of the $1$s in the tableaux must 
 go in the first (bottom) row, and they fill it completely. Next, we need to fill $\alpha_2$ cells with $2$. Here, $2$ will be the lowest entry in any of 
 the rows (other than the first), so any row with a $2$ must have a $2$ as its first entry. Since the first column is strictly increasing, 
 the second row must have a $2$ as its first entry, and no other row may contain any $2$'s. Therefore, we fill the second row entirely 
 with $2$'s, which uses all $\alpha_2$ that we seek to use. Continuing this pattern, we see that the only way to construct a lexical tableau 
 of shape $\alpha$ and type $\alpha$ is to fill each of the $\alpha_i$ cells of row $i$ with $i$'s, and thus  $K_{\alpha,\alpha} = 1$. 
 
 Next, consider some $\alpha <_{\ell} \beta$, meaning there exists some $j$ such that $\alpha_1 = \beta_1, \alpha_2 = \beta_2, \ldots, 
 \alpha_{j-1}=\beta_{j-1},$ and $\alpha_j<\beta_j$. By way of contradiction, suppose that there exists a lexical tableau $T$ of shape $\alpha$ and type $\beta$. Using a similar argument, relative to our preceding argument, the first $j-1$ rows of $T$ are necessarily filled entirely with the integer that matches their row index. That is, if $i < j$ then row 
 $i$ is has exactly $\alpha_i = \beta_i$ each filled with an $i$. Next, we will fill in the $\beta_j$ instances of $j$. Given the current filling, 
 any row with a $j$ must have a $j$ as its first entry. Since the first column is strictly increasing, row $j$ must contain all of these $j$'s. 
 However, there are only $\alpha_j$ cells in row $j$, and we need to place $\beta_j > \alpha_j$ instances of $j$. Thus, we cannot create a 
 lexical tableau of shape $\alpha$ and type $\beta$, meaning $K^{\lx}_{\alpha, \beta}=0$ if $\alpha <_{\ell} \beta$. 
\end{proof}

\begin{table}[h!]\label{tab:trans}
 \centering
 \begin{tabular}{|c|c|c|c|c|c|c|c|c|} \hline
 $K^{\lx}_{\alpha. \beta}$ & (4) & (3,1) & (2,2) & (2,1,1) & (1,3) & (1,2,1) & (1,1,2) & (1,1,1,1) \\ \hline
 (4) & 1 & 1 & 2 & 3 & 1 & 3 & 3 & 6 \\ \hline
 (3,1) & 0 & 1 & 1 & 2 & 1 & 3 & 3 & 6 \\ \hline
 (2,2) & 0 & 0 & 1 & 1 & 1 & 2 & 2 & 3 \\ \hline
 (2,1,1) & 0 & 0 & 0 & 1 & 0 & 1 & 1 & 3 \\ \hline
 (1,3) & 0 & 0 & 0 & 0 & 1 & 1 & 1 & 2 \\ \hline
 (1,2,1) & 0 & 0 & 0 & 0 & 0 & 1 & 1 & 2 \\ \hline
 (1,1,2) & 0 & 0 & 0 & 0 & 0 & 0 & 1 & 1 \\ \hline
 (1,1,1,1) & 0 & 0 & 0 & 0 & 0 & 0 & 0 & 1 \\ \hline
 
 \end{tabular}
 \caption{Values of $K^{\lx}_{\alpha, \beta}$ for $\alpha, \beta \vDash 4$.}
 \label{tab:placeholder}
\end{table}

We can also count standard lexical tableaux with methods coming from the study of immaculate tableaux. Given a cell $c = (i,j)$ in a composition 
 diagram of $\alpha$, a hook of $c$ in $\alpha$, denoted $H_{\alpha}(c),$ is defined to be 

 \[H_{\alpha}(c) = 
 \begin{cases}
 \{(x,y) : i \leq x \leq \ell(\alpha), 1 \leq y \leq \alpha_{x}\}, & \text{ if } j=1,\\
 \{(i,y) : j \leq y \leq \alpha_{i}\}, & \text{otherwise.}
 \end{cases}
 \]
 The hook length $h_{\alpha}(c)$ is defined to be $|H_{\alpha}(c)|.$
 Berg et al.\ 
 \cite{BergBergeronSaliolaSerranoZabrocki2014}, 
 proved that the number of standard immaculate tableaux of shape $\alpha$ is 
 given by a hook length formula, which we generalize as follows.

\begin{theorem} 
 Let $\alpha \vDash n$. The number of standard lexical tableaux of shape $\alpha$ is given by \[ K^{\lx}_{\alpha, 1^n} = \frac{n! \prod_{i \in [\ell(\alpha)]} (\alpha_i-1)!}{\prod_{c \in \alpha} h_{\alpha}(c)}.\]
\end{theorem}

\begin{proof}
For each standard immaculate tableau of shape $\alpha$, we can generate $\prod_{i \in [k]}(\alpha_i - 1)!$ unique lexical tableaux by permuting the entries within each row, excluding those in the first column. No two lexical tableaux generated this way can be the same, and every lexical tableau is associated with some immaculate tableau in this way, so all will be generated.
\end{proof}

We also find a $q$-hook length formula and a $q$-refinement of Proposition \ref{theoremltstirling}. For $\alpha = (\alpha_1, \ldots, \alpha_k)$, let $\bar{\alpha}=(\alpha_2, \ldots, \alpha_k)$.
Let $SLT(\alpha)$ denote the set of standard lexical tableaux of shape $\alpha$.
Let the bottom-to-top reading word of a lexical tableau $T$, denoted $rw^{\uparrow}(T)$, be the word obtained by concatenating the entries of the rows of $T$ from bottom to top. Let the inversion number of a word $w$ be defined $inv(w) = |\{(i,j): i<j, w_i>w_j\}|$. First, we present a first-row recursion.

\begin{lemma}\label{lem:recursion} Let $X_{\alpha}(q) = \sum_{T \in SLT(\alpha)} q^{inv(rw^{\uparrow}(T))}.$ Then, we have the recursion
    \[X_{\alpha}(q)=[\alpha_1-1]_q! {n-1 \brack \alpha_1-1}_q X_{\bar{\alpha}}(q).\]
\end{lemma}

\begin{proof} For all $T\in SLT(\alpha)$ where $\alpha \vDash n$, the inversions in $rw^{\uparrow}(T)$ correspond to either an inversion completely in the first row, an inversion completely in the rest of the tableaux, or an inversion between an element in the first row and an element in the rest of the tableau. The first row always contains $1$, but the remaining $\alpha_1-1$ entries can be any appropriately sized subset of $[n]$. The $q$-binomial coefficient gives a summand for each subset choice, weighted by the inversions that subset will create between the first row and the rest of the tableau. The term $[\alpha_1-1]_q!$ gives exactly a summand for each possible first row filled with that subset, weighted by the inversions among that row. Finally, $X_{\bar{\alpha}}(q)$ covers the inversions that occur among the rest of the tableau for all possible options.    
\end{proof}

\begin{theorem}\label{thm:q-hook} For a composition $\alpha$,
    \[\sum_{T\in SLT(\alpha)}q^{inv(rw^{\uparrow}(T))}= \frac{[n]_q! \prod_i[\alpha_i-1]_q!}{\prod_{c\in \alpha}[h_{\alpha}(c)]_q}.\]
\end{theorem}

\begin{proof}
First, observe that, by definition, for $c \in \bar{\alpha}$, we have $h_{\bar{\alpha}}(c)=h_{\alpha}(c)$. Further, $h_{\alpha}((1,1))=n$ and $h_{\alpha}((1,i))=\alpha_1-i+1$ for $1<i \leq \alpha_1$. Thus,$\prod_{c\in \alpha}[h_{\alpha}(c)]_q=[n]_q[\alpha_1-1]_q!\prod_{c \in \bar{\alpha}}[h_{\bar{\alpha}}(c)]_q$. We now apply our recursion from Lemma \ref{lem:recursion} to prove our formula.
    \begin{align*}
        \sum_{T\in SLT(\alpha)}q^{inv(rw^{\uparrow}(T))} &=  [\alpha_1-1]_q! {n-1 \brack \alpha_1-1}_q \sum_{S\in SLT(\bar{\alpha} )}q^{inv(rw^{\uparrow}(S))}\\
        &=  [\alpha_1-1]_q! \frac{[n-1]_q!}{[n-\alpha_1]_q![\alpha_1-1]_q!} \frac{[n-\alpha_1]_q! \prod_{i \geq 2} [\alpha_i-1]_q!}{\prod_{c \in \bar{\alpha}}[h_{\bar{\alpha}}(c)]_q}\\
        &=\frac{[n]_q[n-1]_q! \prod_{i \geq 1} [\alpha_i-1]_q!}{[n]_q[\alpha_1-1]_q!\prod_{c \in \bar{\alpha}}[h_{\bar{\alpha}}(c)]_q}\\
        &= \frac{[n]_q!\prod_{i \geq 1}[\alpha_i-1]_q!}{\prod_{c \in \alpha}[h_{\alpha}(c)]_q}. \qedhere
    \end{align*}
\end{proof}

Let $c_q[n,k]$ denote the $q$-Stirling numbers of the first kind \cite{Gould1961qStirling}, defined  $c_q[n,k] = c_q[n-1,k-1]+[n-1]_q \cdot c_q[n-1,k]$ where $c_q[n,0] = \delta_{n,0}$.

\begin{proposition}\label{cor:stir} For $n,k \in \N$ with $k \leq n$,
    \[\sum_{\substack{\alpha \models n,\\ \ell(\alpha)=k}} \sum_{T \in SLT(\alpha)}q^{inv(rw^{\uparrow}(T))} = c_q[n,k].\]
\end{proposition}

\begin{proof}
Consider a standard lexical tableau $T$ of size $n$ with $k$ rows.  If entry $n$ is in a  row by itself, necessarily the last (top) row, then removing it yields a SLT of size $n-1$ with $k-1$ rows and an unchanged number of inversions in the reading word. If the entry $n$ is not in a row by itself, then $T$ is obtained by inserting an entry $n$ into any spot (other than the first) in the reading word of some SLT of size $n-1$ with $k$ rows. Each possible spot in the reading word corresponds to a unique spot $n$ could be placed in the tableau, since $n$ could not be in the first column. Inserting $n$ before the final $i$ entries contributes $i$ inversions to $T$. The possible inversions added have a generating function of $[n-1]_q$.  This recursive structure on standard lexical tableaux matches that of the q-Stirling numbers.
\end{proof}

Additionally, note that the formula in Theorem \ref{thm:q-hook} is a special case of the Björner-Wachs $q$-hook length formula for forests \cite{BJORNER1989165} applied to certain caterpillar trees. 



\subsection{Maj-divisible Tableaux, standardization, and descents}\label{sec:maj}

We can extend the standardization map on immaculate tableaux to lexical tableaux. If $T$ is a lexical tableau of shape $\alpha$, then $std(T)=S$ will be the standard lexical tableau of shape $\alpha$ created as follows. Read through the entries in $T$ starting first with those equal to 1, then 2, etc. Among all entries of the same value in $T$, read  within each row from left to right, reading rows from top to bottom. Replace entries in the order they are read, starting with 1 and increasing each entry. Note that, like with immaculate tableaux, no two lexical tableaux of the same shape and type can have the same standardization.

\ytableausetup {boxsize=normal,notabloids}
\begin{example} The lexical tableaux that standardize to $S$ below are exactly $S$, $T_1$, and $T_2$:
 \[ S = \begin{ytableau}
 1 & 2 & 4 & 3
 \end{ytableau}\ , \qquad \quad T_1 = \begin{ytableau}
 1 & 2 & 3 & 2
 \end{ytableau}\ , \quad T_2 = \begin{ytableau}
 1 & 1 & 3& 2
 \end{ytableau}. \] 
\end{example}

Unlike in the immaculate case, the set of types of lexical tableaux that standardize to a fixed tableau $S$ does not always form an interval under refinement order. This property is crucial to the expansion of dual immaculate functions into the fundamental basis. Our lexical analogue requires a replacement, which we develop here.

Let the top-to-bottom reading word of a tableau $T$, denoted $rw^{\downarrow}(T)$, be the word obtained by concatenating the rows of $T$ starting with the top row and moving down.
The descent set of a word $w = w_1w_2 \cdots w_n$ is defined as $Des(w) = \{ i : w_i > w_{i+1}\}$. The major index of a word $w$ is defined as $maj(w)=\sum_{i\in Des(w)} i$. The inverse descent set of a permutation $\sigma = \sigma_1 \sigma_2 \cdots \sigma_n$ is defined $iDes(\sigma) = \{i: \sigma^{-1}(i)>\sigma^{-1}(i+1)  \}$.
We say that a standard tableau $T$ has a descent at $i$ if $i+1$ is in a strictly higher row than $i$ or sits to the left of $i$ in the same row. The descent composition of $T$ with $\{d_1, \ldots, d_k\}$ is given by $Comp(Des(T))=(d_1, d_2-d_1, d_3-d_2, \ldots, d_k-d_{k-1}, n-d_k)$.
Note that the descent set of $T$ corresponds to the inverse descent set of $rw^{\downarrow}(T)$.

We define a  \emph{maj-divisible tableau} of shape $\alpha$ as a filling of a diagram of shape $\alpha$ with positive integers such that the row minima strictly increase from bottom to top and the major index of row $i$ is divisible by the length of row $i$ for each $i\in [\ell(\alpha)]$. We denote the set of maj-divisible tableaux of shape $\alpha$ and type $\beta$ by $MD(\alpha,\beta)$, and let $MD(\alpha) = \cup_{\beta} MD(\alpha,\beta)$. If such a tableau $T$ of size $n$ contains exactly the entries $1$ through $n$ we call it standard. 
Let $SMD(\alpha,\gamma)$ denote the set of standard maj-divisible tableaux of shape $\alpha$ with descent composition $\gamma$.
The standardization of a maj-divisible tableau $T$ of size $n$, denoted $std(T)$, is obtained by relabeling $T$ with entries $1$ through $n$ in order of the entries of $T$, where $rw^{\downarrow}$ order is used among entries of the same value.

In \cite[Theorem 1.7]{ahlbach2018cyclicsievingnecklacesbranching}, Ahlbach and Swanson prove that the number of necklaces of length $k$ and content $\beta$ equals the number of words with content $\beta$ with major index divisible by $k$. Their argument uses cyclic sieving to rewrite necklace generating functions. We take advantage of this relationship, which implies $K^{\lx}_{\alpha, \gamma} = |MD(\alpha,\gamma)|$, to describe certain coefficients to be used later in our fundamental basis expansion.

\begin{theorem}\label{thm:K-to-J} Let $J^{\lx}_{\alpha,\beta} = |SMD(\alpha,\beta)|$ for all $\alpha, \beta \vDash n$. For $\alpha, \gamma \models n \in \N$,  we have 
\[ K^{\lx}_{\alpha, \gamma} = \sum_{\beta \succeq \gamma} J^{\lx}_{\alpha,\beta}.\]
\end{theorem}

\begin{proof}
    Since $K^{\lx}_{\alpha, \gamma} = |MD(\alpha,\gamma)|$ and $\sum_{\beta \succeq \gamma} J^{\lx}_{\alpha, \beta} = |\sqcup_{\beta \succeq \gamma} SMD(\alpha, \beta)|$, 
    it suffices to show that $std:MD(\alpha, \gamma) \rightarrow \sqcup_{\beta \succeq \gamma} SMD(\alpha,\beta)$ is a bijection. 
    First, we verify that for $T \in MD(\alpha,\gamma)$, the tableau $S=std(T)$ is in $SMD(\alpha,\beta)$. 
    Observe that every row of $S$ has the exact same descent set as the corresponding row of $T$ since strict inequalities between adjacent entries are preserved and two equal adjacent entries are filled left to right, creating an ascent. 
    Thus, the major index of each row in $S$ is the same as the major index of each row in $T$. 
    This verifies that the major index of each row of $S$ is divisible by the row length. Since the row minima strictly increase in $T$, the row minima will strictly increase in $S$. 
    Since equal entries in $T$ are standardized in top-to-bottom reading order, we will not have $j \in iDes(rw^{\downarrow}(S))$ if $j \not\in \{\gamma_1, \gamma_1+ \gamma_2, \ldots, \gamma_1 + \gamma_2 + \cdots + \gamma_{\ell-1}\}$. Therefore, $Comp(iDes(rw^{\downarrow}(S)))$ coarsens $\gamma$. 
    We conclude that $S \in \sqcup_{\beta \succeq \gamma}SMD(\alpha,\beta)$. It is straightforward to check that no two maj-divisible tableaux of the same shape and type standardize to the same thing,  and so $std$ is injective from $MD(\alpha,\gamma)$ to $\sqcup_{\beta \succeq \gamma} SMD(\alpha,\beta)$.

    It remains to show that for every standard maj-divisible tableau $S$ of shape $\alpha$ with descent composition $\beta$ where $\beta \succeq \gamma$, there exists a maj-divisible tableau $T$ of shape $\alpha$ and type $\gamma$. 
    Consider such an $S$, and replace the entries in $S$ in order with entries of type $\gamma$ (first ones, then twos, etc.) to obtain a new tableau $T$ of type $\gamma$. 
    If $a > b$ are entries of $S$ and $a$ appears before $b$ in $rw^{\downarrow}(S)$,
then some $j$ with $b \leq j < a$ is an inverse descent of $rw^{\downarrow}(S)$. Thus, this $j$ is a break of $\gamma$, so $a$ and $b$ receive different values in $T$.
Applying this to adjacent entries $a > b$ in a row shows that row descents are
preserved. No new descents arise since the relabeling is weakly order
preserving. Thus each row of $T$ has major index divisible by its length.
Applying it to the minima $m_{i+1} > m_i$ of rows $i+1$ and $i$ shows that the
row minima of $T$ strictly increase. Finally, entries of $S$ that receive the
same value in $T$ appear in increasing order in $rw^{\downarrow}(S)$, so $std(T)=S$.

    We have shown that there exists a $T \in MD(\alpha,\gamma)$ such that $std(T)=S$ for any such $S$.  
    Thus, $std$ is surjective from $MD(\alpha,\gamma)$ to $\sqcup_{\beta \succeq \gamma} SMD(\alpha,\beta)$
\end{proof}

\begin{example} There are three lexical tableaux of shape $(3,1)$ and type $(1,2,1)$: 

\[
\begin{ytableau}
    3\\
    1 & 2 & 2
\end{ytableau} \qquad \qquad
\begin{ytableau}
    2\\
    1 & 2 & 3\\
\end{ytableau} \qquad \qquad
\begin{ytableau}
    2\\
    1 & 3 & 2
\end{ytableau}\vspace{2mm}
\]
Likewise, there are three maj-divisible tableaux of shape $(3,1)$ and type $(1,2,1)$, which correspond to three standard maj-divisible tableaux of shape $(3,1)$ with descent compositions that coarsen $(1,2,1)$:
\[\begin{ytableau}
    3\\
    1 & 2 & 2
\end{ytableau} \rightarrow \ \begin{ytableau}
    4\\
    1 & 2 & 3
\end{ytableau}\qquad \qquad
\begin{ytableau}
    2\\
    1 & 2 & 3
\end{ytableau} \rightarrow \
\begin{ytableau}
    2\\
    1 & 3 & 4
\end{ytableau}
\qquad \qquad
\begin{ytableau}
    2\\
    3 & 2 & 1
\end{ytableau} \rightarrow \ 
\begin{ytableau}
    2\\
    4 & 3 & 1
\end{ytableau}
\]
The three standard tableaux above have descent compositions $(3,1)$, $(1,3)$, and $(1,2,1)$, respectively.
\end{example}

\section{Lexical functions in QSym and NSym}\label{sectionqsym}

Let $LT_{\alpha}$ denote the set of all lexical tableaux of shape $\alpha$. Given $T \in LT_{\alpha}$, let $type(T)$ denote the
 type of $T$. If $type(T)=\beta= (\beta_1, \beta_2, \ldots, \beta_k)$, we associate $T$ with the monomial $x^T = x_1^{\beta_1} 
 x_2^{\beta_2} \cdots x_k^{\beta_k}$. 

\begin{definition}
 For a composition $\alpha \vDash n$, define the dual lexical function by \[ \lx^*_{\alpha} = \sum_{T \in LT_{\alpha}} x^T.\]
\end{definition}

Note that we define these as the \emph{dual} lexical functions so that we may call their duals simply the lexical functions. This is consistent with the immaculate and dual immaculate functions of $\textsf{NSym}$ and $\textsf{QSym}$.

The dual lexical quasisymmetric functions can also be constructed using Grinberg's strict-minimum restricted product \cite{Grinberg_2017}.  Let $N_m$ be the weight generating function for necklaces of length $m$, or equivalently, the necklace character $ch(Ind_{C_m}^{S_m}(\mathbf{1}))$. Note that these $N_m$ appear as the coefficients of $q^m$ in the Kra{\'s}kiewicz-Weyman symmetric function $KW_m(x;q)$ introduced in \cite{ahlbach2018cyclicsievingnecklacesbranching}.

\begin{theorem}\label{thm:neck} For $\alpha \models n$,
    \[\lx^*_{\alpha} = N_{\alpha_1} \prec (N_{\alpha_2} \prec ( \cdots \prec (N_{\alpha_k} \prec 1) \cdots)).\]
\end{theorem}

\begin{proof} First, note that $\mathfrak{m} \prec 1 = \mathfrak{m}$.
    It suffices to prove that $\lx^*_{\alpha} = N_{\alpha_1} \prec \lx^*_{(\alpha_2, \ldots, \alpha_{\ell})}$. Consider the restricted product $\prec$ of two monomials $x^w$ from $N_{\alpha_1}$, where $w$ is a necklace word, and $x^T$ from $\lx^*_{(\alpha_2, \ldots, \alpha_{\ell})}$, where $T$ is a lexical tableau.  By definition, $min\ supp(x^T)$ is the smallest entry (also the first entry) in the first row of $T$. So, $x^w \prec x^T=0$ if the minimum value in $w$ is greater than or equal to the minimum value in $T$.  If the minimum value in $w$ is less than the minimum value in $T$, then $x^w \prec x^T=x^wx^T$. This is exactly the monomial obtained by prepending a row filled with word $w$ onto the tableau $T$, and it only occurs when such an action produces a valid lexical tableau. Moreover, every lexical tableau is created this way.
\end{proof}

As Frobenius characteristics of representations, the necklace characters $N_m$ are Schur positive and thus fundamental positive. One can prove that the operation $\prec$ preserves fundamental positivity.   This yields an algebraic proof that the dual lexical functions are fundamental positive.

\begin{remark}\label{rem:GR}
In \cite{GesselReutenauer1993}, Gessel and Reutenauer develop a bijection between multisets of primitive necklaces and words with enumerative and representation-theoretic applications. Specifically, note that the weight generating function for primitive necklaces of size $n$, denoted $Lie_n$, is also the symmetric function associated to the Lie representation of $S_n$ of type $(n)$. 
One may be inclined to define tableaux whose rows correspond to Lyndon words, which are the lexicographically least representations of a primitive necklace. The corresponding quasisymmetric functions would be equal to those obtained by replacing $N_m$ with $Lie_m$ in Theorem \ref{thm:neck}. 
These functions, however, would not form a basis. Observe that the only Lyndon words of length two are of the form $ab$ where $a<b$.  These, and thus the tableaux of shape $(2)$, correspond exactly to the tableaux of shape $(1,1)$ which would need to be filled with a pair $a < b$. Thus, $Lie_2 = Lie_1 \prec Lie_1$.
\end{remark}

\subsection{Basis expansions}

 Given a lexical tableau $T$ with entries given by the set $N$ with $|N|=n$, let $pack(T)$ be the lexical tableau that is
 obtained by replacing the entries from $N$ in $T$ with entries in $[n]$ according to the unique order-preserving bijection between $N$ and $[n]$. We call a lexical tableau \emph{packed} if $pack(T) = T$, or in other words, its entries correspond exactly to the elements of the set $[n]$. These are exactly the lexical tableaux whose type is given by a strong composition, as opposed to a weak composition.

\begin{theorem}\label{theoremlxstarM}
The dual lexical functions have a positive, uni-triangular expansion in terms of the monomial quasisymmetric functions as \[ \lx_{\alpha}^* = \sum_{\beta \vDash n} K^{\lx}_{\alpha, \beta} M_{\beta}.\] 
\end{theorem}

\begin{proof} 
Observe that, for a packed lexical tableau $T_p$ of type $\beta$, we have $\sum_{pack(T) = T_p} x^T = M_{\beta}$. Then, \[\sum_{T \in LT_{\alpha}} x^T = \sum_{\substack{T \in LT_{\alpha},\\ pack(T) = T}} M_{type(T)} = \sum_{\beta \vDash 
 n} K^{\lx}_{\alpha,\beta} M_{\beta}. \qedhere \] 
\end{proof}

\begin{example}

We have the expansion \[\lx^*_{(3,1)} = M_{(3,1)} + M_{(2,2)} + 2M_{(2,1,1)} + M_{(1,3)} + 3M_{(1,2,1)} + 3M_{(1,1,2)} + 6M_{(1,1,1,1)},\]

corresponding to the lexical tableaux

\[ \begin{ytableau}
 2\\
 1 & 1 & 1 
\end{ytableau} \quad \begin{ytableau}
 2\\
 1 & 1 & 2 
\end{ytableau} \quad \begin{ytableau}
 3\\
 1 & 1 & 2 
\end{ytableau} \quad \begin{ytableau}
 2\\
 1 & 1 & 3
\end{ytableau} \quad
\begin{ytableau}
 2\\
 1 & 2 & 2 
\end{ytableau} \quad
\begin{ytableau}
 2\\
 1 & 3 & 3
\end{ytableau}\] \vspace{2mm}
\[
\begin{ytableau}
 3\\
 1 & 2 & 3
\end{ytableau} \quad
\begin{ytableau}
 3\\
 1 & 3 & 2
\end{ytableau}
\quad 
\begin{ytableau}
 3\\
 1 & 2 & 2
\end{ytableau} \quad \begin{ytableau}
 2\\
 1 & 2 & 3
\end{ytableau} \quad \begin{ytableau}
 2\\
 1 & 3 & 2
\end{ytableau} \]
\vspace{2mm}
\[
\begin{ytableau}
 4\\
 1 & 2 & 3
\end{ytableau} \quad \begin{ytableau}
 4\\
 1 &3 & 2
\end{ytableau} \quad \begin{ytableau}
 3\\
 1 & 2 & 4
\end{ytableau} \quad
\begin{ytableau}
 3\\
 1 & 4 & 2
\end{ytableau} \quad
\begin{ytableau}
 2\\
 1 & 3 & 4
\end{ytableau} \quad
\begin{ytableau}
 2\\
 1 & 4 & 3
\end{ytableau}.
\] \vspace{0mm}
\end{example}

By Proposition \ref{theorem:lexical_coeff}, we know the dual lexical functions are a basis of $QSym$.

\begin{corollary}
 $\{\lx^*_{\alpha} : \alpha \vDash 
 n\}$ is a basis for $\emph{\textsf{QSym}}_n$ and $\bigcup_{n \geq 0}\{ \lx^*_{\alpha} : \alpha \vDash
 n\}$ is a basis for $\emph{\textsf{QSym}}$.
\end{corollary}

\begin{remark}
 Observe that if $\alpha \vDash n$ and $\alpha_i \leq 2$ for all $i \in [\ell(\alpha)]$, then $\lx^*_{\alpha} = \mathfrak{S}^*_{\alpha}$. 
 This may be seen using the property that a necklace word of length two is always a weakly increasing word. Thus, the lexical 
 tableaux of shape $\alpha$ where $\alpha_i \leq 2$ for all $i \in \ell(\alpha)$ are exactly the immaculate tableaux of shape $\alpha$.
\end{remark}

We define lexical functions in $\textsf{NSym}$ using duality with the dual lexical functions.

\begin{definition}\label{definitionlexinner}
 Define the lexical basis of $\textsf{NSym}$ to be the unique basis $\bigcup_{n \geq 0}\{\lx_{\alpha} : \alpha \vDash 
 n\}$ such that 
 \begin{equation}\label{innerforlex}
 \langle \lx_{\alpha}, \lx^{*}_{\beta} \rangle = \delta_{\alpha,\beta},
 \end{equation} 
 for all compositions $\alpha,\beta \vDash 
 n$ for all $n \geq 0$.
\end{definition}
 
From the inner product relation in \eqref{innerforlex}, we obtain the following expansion.

\begin{corollary}\label{Htolex}
 For $\beta \vDash n$, we have the expansion \[H_{\beta} = \sum_{\alpha 
 \geq_{\ell} \beta} K^{\lx}_{\alpha, \beta} \lx_{\alpha}\]
\end{corollary}

 Using  Theorem \ref{thm:K-to-J}, we can also give a positive expansion of the ribbon basis of $\textsf{NSym}$ into the lexical functions, and dually a positive expansion of the dual lexical functions into the fundamental basis.

\begin{theorem}\label{Rtolex}
 We have the dual positive expansions 
 \[R_{\beta} = \sum_{\alpha \models n} J^{\lx}_{\alpha, \beta} \lx_{\alpha} \qquad \text{and } \qquad \lx^*_{\alpha}= \sum_{\beta \models n} J^{\lx}_{\alpha,\beta}F_{\beta}.\]
\end{theorem}

\begin{proof}
 By Corollary \ref{Htolex}, we have $H_{\gamma} = \sum_{\alpha \models n} K^{\lx}_{\alpha, \gamma} \lx_{\alpha}$. Thus, using Theorem \ref{thm:K-to-J},
 \[H_{\gamma} = \sum_{\alpha \models n} K^{\lx}_{\alpha,\gamma} \lx_{\alpha} = 
 \sum_{\alpha \models n} \sum_{\beta \succeq \gamma} J^{\lx}_{\alpha, \beta} \lx_{\alpha} = \sum_{\beta \succeq \gamma} \left( \sum_{\alpha \models n} J^{\lx}_{\alpha, \beta}\lx_{\alpha} \right)\]
 Given that $H_{\gamma} = \sum_{\beta \succeq \gamma} R_{\beta}$, it must be that $R_{\beta} = \sum_{\alpha \models n} J^{\lx}_{\alpha, \beta}\lx_{\alpha}$.  The expansion in $QSym$ follows by duality.
\end{proof}
 
\begin{corollary}\label{cor:EtoL} For $n \in \N$, 
\[E_n = \sum_{\substack{\alpha \vDash n \\ \text{every part of } \alpha \text{ is odd}}} \lx_{\alpha}.\]
\end{corollary}

\begin{proof}
Recall that $E_n = R_{(1)^n}$, so by Theorem \ref{Rtolex}, $E_n = \sum_{\alpha \vDash n} J^{\lx}_{\alpha,(1)^n} \lx_{\alpha}$. 
So, the coefficients in the expansion count the number of standard maj-divisible tableaux of shape $\alpha$ with every possible descent.  
Such a standard maj-divisible tableaux will be strictly decreasing in each row. 
A decreasing row of length $r$ will have major index $\frac{r(r-1)}{2}$, which is divisible by $r$ if and only if $r$ is odd. 
Thus, the coefficient $J^{\lx}_{\alpha, (1)^n}$ is 1 when every part of $\alpha$ is odd and 0 otherwise.   
\end{proof}

We take this further to a positive expansion of elementary noncommutative symmetric functions into the lexical basis. The complement of an integer composition $\gamma = (\gamma_1, \ldots, \gamma_k)$ is given by $\gamma^c = Comp([n-1] \setminus Set(\gamma))$ where $Set(\gamma) = \{\gamma_1, \gamma_1 + \gamma_2, \ldots, \gamma_1 + \cdots + \gamma_{k-1}\}$.

\begin{proposition}\label{prop:EtoL}
    For $\beta \vDash n$,
    \[ E_{\beta} = \sum_{\alpha \vDash n} \left( \sum_{\gamma^c \succeq \beta}J^{\lx}_{\alpha, \gamma} \right) \lx_{\alpha}.\]
\end{proposition}

\begin{proof} Given the general expansion of the elementaries in $NSym$ to the ribbon basis, along with Theorem \ref{Rtolex}, we have
    \[E_{\beta} = \sum_{\gamma^c \succeq \beta} R_{\gamma} = \sum_{\gamma^c \succeq \beta} \left(\sum_{\alpha\vDash n} J^{\lx}_{\alpha, \gamma} \lx_{\alpha} \right) = \sum_{\alpha \vDash n} \left( \sum_{\gamma^c \succeq \beta}J^{\lx}_{\alpha, \gamma} \right) \lx_{\alpha}.\qedhere\]
\end{proof}

The dual lexical functions and lexical functions do not expand positively into most other bases of $QSym$ and $NSym$, respectively.

\begin{example}\label{ex:lf} Small expansions of (dual) lexical functions reveal negative coefficients:
\begin{align*}
    &\lx_{4}^*= \mathfrak{S}^*_{4} + \mathfrak{S}^*_{22} + \mathfrak{S}^*_{211} - \mathfrak{S}^*_{13}, \qquad \lx^*_{11}=\frac{1}{2}\Psi_{11}-\frac{1}{2}\Psi_{2} = \frac{1}{2}\Phi_{11}-\frac{1}{2}\Phi_{2},\\
&\lx^*_{23}=\shin^*_{23}+\shin^*_{2111}+\shin^*_{14}+\shin^*_{131}-\shin^*_{122}-\shin^*_{113}+\shin^*_{1121}-\shin^*_{1112}\\
&\lx_{111} =  R_{111}-R_3 = \mathfrak S_{111}-\mathfrak S_{3} = \shin_{111}-\shin_{3}, \qquad \lx_{11} = \frac{1}{2}\mathbf{\Psi}_{11}-\frac{1}{2}\mathbf{\Psi}_{2} =\frac{1}{2}\mathbf{\Phi}_{11}-\frac{1}{2}\mathbf{\Phi}_2
\end{align*}

\end{example}

Data for $n\leq 9$ suggests, however, that these functions have positive expansions into the Young quasisymmetric Schur functions.

\begin{conjecture}\label{conj:yqs}
    Dual lexical functions expand positively into the Young quasisymmetric Schur basis.
\end{conjecture}

\subsection{Antipodes of lexical basis elements}
 Given an open problem concerning the immaculate basis, one might consider a corresponding problem for the lexical basis, in the hope that the use of lexical basis elements could shed light on the original problem, and this provides a main source of 
 motivation concerning the study of the $\lx$-basis. In this direction, the problem of determining cancellation-free formulas for the 
 antipode $S = S_{\textsf{NSym}}$ mapping of $\textsf{NSym}$ evaluated at immaculate basis elements remains open, despite past progress 
 on this problem \cite{BenedettiSagan2017,Campbell2023}. Since progress on this problem has been made for cases given by specific 
 families of composition shapes, we consider the problem 
 of determining (ideally, cancellation-free) formulas for lexical functions indexed by the same composition shapes. For example, the formula for the antipode of a lexical function indexed by a single row follows from our Corollary \ref{cor:EtoL} and the fact that $\lx_{(n)}=H_{(n)}$.

 \begin{corollary}
     For $n \in \N$,
     \[S(\lx_{(n)})=(-1)^n \sum_{\substack{\alpha \vDash n \\ \text{every part of } \alpha \text{ is odd}}} \lx_{\alpha}.\]
 \end{corollary}

 In their seminal paper on antipodes and involutions, Benedetti and Sagan \cite{BenedettiSagan2017} used a bijective approach toward 
 obtaining formulas for expanding $S(\mathfrak{S}_{\alpha})$ in the $\mathfrak{S}$-basis, for the cases whereby 
 $\alpha$ is a hook or consists of two rows. Subsequently, cancellation-free formulas for expanding $S(\mathfrak{S}_{\alpha})$ in the 
 $R$-basis were determined for the cases whereby $\alpha$ is a rectangle or certain products of rectangles \cite{Campbell2023}, and this 
 was later generalized by Allen and Mason \cite{AllenMason2025}. The $S(\mathfrak{S}_{\alpha})$-to-$R$ expansion formulas relied on 
 the Jacobi--Trudi formula for the immaculate basis, and the known formula for the antipode of an immaculate-hook also relied on the 
 Jacobi--Trudi formula for the $\mathfrak{S}$-basis, but the $\lx$-basis does not seem to satisfy any such determinantal formula. This 
 leads us to consider the problem of evaluating $S(\lx_{\alpha})$ for the case whereby $\alpha$ is a two-rowed composition. 

\begin{example}
 From the expansion 
\begin{equation}\label{lx24}
 \lx_{24} = H_{24} - H_{33} - H_{42} + H_{51} + 3 H_{6}, 
\end{equation}
 we obtain a recursive formula for $S(\lx_{24})$ by applying $S$ to both sides of \eqref{lx24} and using the property that $S$ is an 
 anti-homomorphism together with the relation $S(H_{n}) = (-1)^n E_n$. 
\end{example}

 With regard to the following lemma, we adopt the notational convention whereby: For a composition $\alpha$, the concatenation 
 $\alpha \cdot (0)$ may be identified with $\alpha$. 

\begin{theorem}\label{mainccoeff}
 For positive integers $a$ and $b$, define $$c^{(a, b)}_{1} := -K^{\lx}_{(a+1,b-1), (a, b)}$$ and then recursively define $$ c^{(a, 
 b)}_{i} := -\left( K^{\lx}_{(a+i,b-i), (a, b)} + \sum_{j=1}^{i-1} c^{(a, b)}_{j} K^{\lx}_{(a+i,b-i),(a+j,b-j)} \right) $$ for all 
 possible indices $i$. Then 
\begin{equation}\label{lxtworowexp} 
 \lx_{(a, b)} = H_{(a, b)} 
 + c_{1}^{(a, b)} H_{(a+1, b-1)} 
 + c_{2}^{(a, b)} H_{(a+2,b-2)}
 + \cdots + c_{b}^{(a, b)} H_{(a+b)}. 
\end{equation}
\end{theorem}

\begin{proof}
 From the condition whereby lexical tableaux are required to be strictly increasing in the first column, we may deduce, from Corollary 
 \ref{Htolex}, that each $\lx$-term in the expansion of $H_{(a, b)}$ is indexed by a composition of length not exceeding $2$. Moreover, 
 from the triangularity of the transition matrices between the $\lx$- and $H$-bases given by Proposition \ref{theorem:lexical_coeff}, we have 
 that   the compositions indexing the $H$-elements in the $H$-expansion of $\lx_{(a, b)}$ are of the form $\lx_{(a+i,b-i)}$ for nonnegative 
 $i$.   The desired result then follows inductively, by expanding $H_{(a, b)}$ into the $\lx$-basis and performing an equivalent form of 
 Gaussian elimination. 
\end{proof}

 As a consequence, we obtain the antipode formula 
 \begin{equation}\label{antipodeE}
 S\big(\lx_{(a, b)}\big) = (-1)^{a+b} \big( E_{(b, a)} 
 + c_{1}^{(a, b)} E_{(b-1, a + 1)} 
 + c_{2}^{(a, b)} E_{(b-2, a + 2)}
 + \cdots + c_{b}^{(a, b)}
 E_{(a+b)} \big), 
 \end{equation} 
 i.e., by applying the antipode map $S$ to both sides of the expansion in \eqref{lxtworowexp}. In view of Benedetti and Sagan's bijective 
 approach toward the cancellation-free evaluation of antipodes of the form $S(\mathfrak{S}_{(a, b)})$, this motivates the problem of 
 finding combinatorial interpretations for coefficients of the form $c_{i}^{(a, b)}$. 

 Adopting notation from the OEIS entry {\tt A047996}, let $T(n, k)$ denote the number of necklaces with $k$ black beads and $n-k$ white 
 beads. Equivalently, 
 we may define $T(n, k)$ so that 
\begin{equation}\label{Texplicit} 
 T(n, k) := \frac{1}{n}
 \sum_{d \mid n, k} \varphi(d) \binom{n/d}{k/d}. 
\end{equation} 

\begin{lemma}\label{relationT}
 The relation $$ K^{\lx}_{(a+i,b-i), (a+j,b-j)} = T(a+i,i-j) $$ holds for $j \in \{ 0, 1, \ldots, i-1 \}$. 
\end{lemma}

\begin{proof}
 Either side of \eqref{Texplicit} is equal to the number of length-$n$ binary words $w$ with $k$ entries equal to $1$ such that $w$ is 
 lexicographically less than or equal to all of the cyclic permutations of $w$. A lexical tableau of shape $(a+i,b-i)$ and content $(a+j,b-j)$ 
 is uniquely determined by its initial row, which, when read as a word, is lexicographically less than or equal to all of the cyclic 
 permutations of the same word. 
\end{proof} 

 From \eqref{Texplicit} and Lemma \ref{relationT}, we thus obtain an explicit formula for the antipode $ S\big(\lx_{(a, 
 b)}\big)$, with coefficients given recursively in terms of Euler's totient function, as below. 

\begin{theorem}\label{mainantipodethm}
 For positive integers $a$ and $b$, let 
 $\mathcal{C}^{(a, b)}_{1} := -1 $ 
 and let 
\begin{equation}\label{displayfinalcoef}
 \mathcal{C}^{(a, b)}_{i} := -\left( 
 T(a+i,i) 
 + \sum_{j=1}^{i-1} \mathcal{C}^{(a, b)}_{j}
 T(a+i,i-j) \right) 
\end{equation} 
 for all possible indices $i$. Then 
\begin{equation*}
 S\big(\lx_{(a, b)}\big) = (-1)^{a+b} \big( E_{(b, a)} + \mathcal{C}_{1}^{(a, b)} E_{(b-1, a + 1)} + \mathcal{C}_{2}^{(a, b)} E_{(b-2, a + 2)}
 + \cdots + \mathcal{C}_{b}^{(a, b)} E_{(a+b)} \big). 
 \end{equation*} 
\end{theorem}

\begin{proof}
 This follows from Theorem \ref{mainccoeff} and \eqref{antipodeE} and Lemma \ref{relationT}. 
\end{proof}

\begin{example}
 We obtain the antipode evaluation $$ S\big( \lx_{33} \big) = E_{33} - E_{24} - E_{15}$$ using Theorem \ref{mainantipodethm}, and similarly 
 for $$ S\big( \lx_{42} \big) = E_{24} - E_{15} - 2 E_{6}. $$
\end{example}

 A similar approach, relative to our proof of Theorem \ref{mainantipodethm}, can be used to obtain formulas for the 
 antipodes of $\lx$-basis indexed by compositions with a fixed number of rows greater than $2$, but, as is the case with the immaculate 
 basis, this leads to more and more complicated formulas, and we encourage the exploration of such higher-order antipode formulas. It would be desirable to apply Theorem \ref{mainantipodethm} to obtain a combinatorial interpretation for the coefficients of the form 
 shown in \eqref{displayfinalcoef}, and to apply Theorem \ref{mainantipodethm} to obtain a cancellation-free 
 expansion of $S(\lx_{(a, b)})$ in the $\lx$-basis.

\subsection{Products with lexical functions}\label{sec:Prod}
 The difficulties, from both computational and combinatorial points of view, 
 associated with problems related to the evaluation of $\lx$-basis elements are reflected in how
 the $\lx$-basis
 does not satisfy a Pieri rule with $\{ 0, 1 \}$-coefficients, as illustrated below. In fact, there is no positive Pieri rule at all, as evidenced by the product $\lx_{13}H_1= \lx_{131}+\lx_{14}+\lx_{23}+\lx_{32}-\lx_5$.

\begin{example}
  Consider the expansion 
\begin{multline*}
 \lx_{{{\tikztableausmall{ {X}, {X, X}} }}} 
 H_{{{\tikztableausmall{ {\boldentry X, \boldentry X}} }}} = 
 \lx_{{{\tikztableausmall{ {\boldentry X, \boldentry X}, {X}, {X, X}} }}} + \lx_{{{\tikztableausmall{ {\boldentry X}, {X, \boldentry X}, {X, X}} }}} + \lx_{{{\tikztableausmall{ {X, \boldentry X, \boldentry X}, {X, X}} }}} + \lx_{{{\tikztableausmall{ {\boldentry X}, {X}, {X, X, \boldentry X}} }}} + \\ 
 \lx_{{{\tikztableausmall{ {X, \boldentry X}, {X, X, \boldentry X}} }}} + 
 4 \lx_{{{\tikztableausmall{ {X}, {X, X, \boldentry X, \boldentry X}} }}} + 
 4 \lx_{{{\tikztableausmall{ {X, X, \boldentry X, \boldentry X, \boldentry X}} }}}.  
\end{multline*} 
\end{example}

 Informally, since lexical tableaux generalize immaculate tableaux, the lexical basis may be seen as more complicated than  
 the immaculate basis, and the fact that the lexical basis does not satisfy a simple Pieri rule may be seen as representative of this. The difficult nature of the $\lx$-basis is such that it does not seem to be feasible to construct this basis using a Jacobi--Trudi-like or determinantal rule, in contrast to the immaculate basis. 
Additionally, products of lexical functions do not expand positively into the lexical basis. The same is true for dual lexical functions.
One can show, however, that there exists a creation operator construction for lexical functions, but it does not have a Bernstein vertex operator factorization.

\section{Column-strict lexical tableaux}\label{sectioncolumnstrict}

We say a lexical tableau is \emph{column-strict} if each column is strictly increasing (including across gaps). This is a natural subset of lexical tableaux to consider given the presence of a column-strict condition on Young tableaux and shin (or extended) tableaux \cite{assaf2018kohnertpolynomials, CampbellFeldmanLightShuldinerXu2014}. 

\begin{example} The following are column-strict lexical tableaux of shape $(3,2)$:

\[
\begin{ytableau}
    4 & 5\\
    1 & 3 & 2
\end{ytableau}
\qquad
\begin{ytableau}
    3 & 5\\
    1 & 4 & 2
\end{ytableau}
\qquad 
\begin{ytableau}
    2 & 5\\
    1 & 4 & 3
\end{ytableau}
\qquad
\begin{ytableau}
    2 & 3\\
    1 & 1 & 2
\end{ytableau}
\qquad
\begin{ytableau}
    3 & 4\\
    1 & 1 & 2
\end{ytableau}
\qquad 
\begin{ytableau}
    2 & 4\\
    1 & 1 & 3
\end{ytableau}
\] \vspace{-2mm}
\end{example}

Let $K^{\clx}_{\alpha,\beta}$ denote the number of column-strict lexical tableaux of shape $\alpha$ and type $\beta$. Since these are a subset of lexical tableaux, we have $K^{\clx}_{\alpha,\beta} \leq K^{\lx}_{\alpha,\beta}$. Additionally, the one lexical tableau of shape $\alpha$ and type $\alpha$ for each composition $\alpha$ is also column-strict. Thus, Proposition \ref{theorem:lexical_coeff} yields the following result.

\begin{corollary}\label{cor:cs_lexical_coeff}
 Given a composition $\alpha \vDash n$, we have $K^{\clx}_{\alpha, \alpha}=1$ and, if $\alpha <_{\ell} \beta$, then $K^{\clx}_{\alpha, \beta}=0$.
\end{corollary}






The number of standard column-strict lexical tableaux of size $n$ in general is given by 
    \[1,\quad 2, \quad 6, \quad 23, \quad 108, \quad 597, \quad3815, \quad 27664, \ldots\]
    This sequence does not appear in the OEIS.

\subsection{Bijections with lattice paths}

We can model standard column-strict lexical tableaux of rectangular shapes using lattice paths. Let an \emph{excursion} be a lattice path that starts and ends at the origin. Let $\mathbf{e}_i \in \Z^{m-1}$ be the vector where all entries are 0 except the $i^{\text{th}}$ entry is equal to $1$.

\begin{proposition}\label{thm:rect_path} For $m \geq 2$, standard column-strict lexical tableaux of shape $(m)^k$ are in bijection with excursions of length $mk$ in the nonnegative $(m-1)$ dimensional orthant with steps $\mathbf{1}$, $-\mathbf{e}_1$, $-\mathbf{e}_2$, $\ldots$, and $-\mathbf{e}_{m-1}$.
\end{proposition}

\begin{proof}
A standard column-strict tableau $T$ of shape $(m)^k$ is associated with the following lattice path. Starting a path at the origin, read each of the entries of $T$ in numerical order. If label $t$ occurs in column $1$, take a $\mathbf{1}$ step, and if label $t$ occurs in column $j >1$, take a $-\mathbf{e}_{j-1}$ step. By construction, this is clearly a lattice path of length $mk$ using the selected steps. Let $c_j(t)$ count the entries of at most $t$ in column $j$. The position after $t$ steps in the path is then $(c_1(t)-c_2(t), \ldots, c_1(t)-c_m(t))$. Each column increases and, in every row, the first entry is the minimum. Thus, the $i^{\text{th}}$ entry in column 1 appears before the $i^{\text{th}}$ entry in every other column, meaning $c_1(t) \geq c_{j}(t)$ for all $j,t$. Thus, the path stays in the nonnegative orthant and ends at the origin.

Inversely, consider an excursion with the stated properties of length $mk$.  Starting at the beginning of the path and moving step by step, place label $t$ in the lowest empty cell of the column associated with step $t$ using the same correspondence as above. This construction automatically forces columns to increase, and ending at the origin ensures that each column receives exactly $k$ entries. Non-negativity ensures that every row's first entry is smaller than all its other entries.
\end{proof}

These excursions are known as \emph{Kreweras excursions} for $m= 3
$. The number of Kreweras excursions of length $3k$ is $\frac{4^k}{(k+1)(2k+1)}\binom{3k}{k}$. For $m\geq 4$, these excursions are known as higher-dimensional Kreweras excursions, and no closed formula is currently known \cite{asinowski2023kreweras, Bostan_2016}. When $m=2$, these excursions correspond to Dyck paths, although not presented in the traditional manner.
They are counted by the Catalan numbers $C_k = \frac{1}{k+1}\binom{2k}{k}$. 

By the same logic, non-rectangular partition shapes correspond to lattice paths as well, but not excursions. Here, given a partition $\lambda$, we let $\lambda_i'$ denote the length of column $i$ in $\lambda$.

\begin{theorem}\label{thm:gen_path}
    For a partition $\lambda = (\lambda_1, \ldots, \lambda_{\ell})$ of $n$, standard column-strict lexical tableaux of shape $\lambda$ are in bijection with lattice paths of length $n$ in the $\lambda_1-1$ dimensional nonnegative orthant that start at the origin and end at $(\lambda_1'-\lambda_2', \lambda_1'-\lambda_3', \ldots, \lambda_1'-\lambda_{\lambda_1}')$ with steps $\mathbf{1}$, $-\mathbf{e}_1$, $-\mathbf{e}_2$, $\ldots$, and $-\mathbf{e}_{\lambda_1-1}$.
\end{theorem}

We also highlight another correspondence to Motzkin paths \cite{StanleyEC2}. A Motzkin path of length $n$ is a lattice path from $(0,0)$ to $(n,0)$ using steps $(1,1)$, $(1,-1)$, and $(1,0)$. 
They are counted by the Motzkin numbers $M_n = \sum_{k=0}^{\lfloor n/2\rfloor}\binom{n}{2k}C_k$. 

\begin{proposition}\label{prop: Motzkin}
    Standard column-strict tableaux of size $n$ whose rows have length of at most 2 are in bijection with Motzkin paths.
\end{proposition}

\begin{proof} A standard column-strict lexical tableau $T$ is associated with a Motzkin path as follows. If entry $i$ in $T$ is in the first column and a row of length two, make step $i$ in the path $(1,1)$. 
If entry $i$ in $T$
is in the second column, make step $i$ in the path $(1,-1)$.
If entry $i$ in $T$
is in the first column and a row of length one, make step $i$ in the path $(1,0)$. This will generate a path that ends at $(n,0)$ and never goes below $y=0$. The inverse map is a straightforward reversal.
\end{proof}

\section{Column-strict lexical functions in $QSym$ and $NSym$}\label{sectionCSLTqsym}

 Let $CSLT_{\alpha}$ denote the set of column-strict lexical tableaux of shape $\alpha$.

\begin{definition}
    For $\alpha \models n$, define the dual column-strict lexical function by \[\clx_{\alpha}^* = \sum_{T \in CSLT_{\alpha}} x^T.\]
\end{definition}

Our next result follows from Theorem \ref{theoremlxstarM}, Proposition \ref{theorem:lexical_coeff},
and the fact that $K^{\clx}_{\alpha, \beta} \leq K^{\lx}_{\alpha, \beta}$ with $K^{\clx}_{\alpha,\alpha}=1$.

\begin{corollary}
    The dual column-strict lexical functions have a positive, unitriangular expansion in terms of the monomial quasisymmetric functions as $\clx_{\alpha}^* = \sum_{\beta} K^{\clx}_{\alpha, \beta}M_{\beta}$. Moreover, the set $\{\clx_{\alpha}^* : \alpha \models n\}$ is a basis for $QSym_n$ and $\cup_{n \geq 0}\{\clx_{\alpha}^* : \alpha \models n\}$ is a basis for $QSym$. 
\end{corollary}

\begin{remark}
    Observe that if $\alpha \vDash n$ and $\alpha_i \leq 2$ for all $i \in [\ell]$, then $\clx^*_{\alpha} = \shin^*_{\alpha}$ where $\shin^*_{\alpha}$ is the extended Schur function of Assaf and Searles \cite{assaf2018kohnertpolynomials}.
The column-strict lexical 
 tableaux of shape $\alpha$ where $\alpha_i \leq 2$ for all $i \in [\ell]$ are exactly the extended tableaux of shape $\alpha$.
\end{remark}

\begin{definition}\label{definitionlexinner}
 Define the column-strict lexical basis of $\textsf{NSym}$ to be the unique basis $\bigcup_{n \geq 0}\{\clx_{\alpha} : \alpha \vDash 
 n\}$ such that 
 \begin{equation}\label{cinnerforlex}
 \langle \clx_{\alpha}, \clx^{*}_{\beta} \rangle = \delta_{\alpha,\beta},
 \end{equation} 
 for all compositions $\alpha,\beta \vDash 
 n$ for all $n \geq 0$.
\end{definition}
 
From the inner product relation in \eqref{cinnerforlex}, we obtain the following expansion.

\begin{corollary}\label{Htocslex}
 For $\beta \vDash n$, we have the expansion \[H_{\beta} = \sum_{\alpha 
 \geq_{\ell} \beta} K^{\clx}_{\alpha, \beta} \clx_{\alpha}\]
\end{corollary}

Data up to $n\leq 8$ suggest that dual column-strict lexical functions expand positively into fundamental quasisymmetric functions.

\begin{conjecture} \label{conj:csF}
    The dual column-strict lexical functions have a positive expansion in terms of the fundamental basis.
\end{conjecture}

The dual column-strict and column-strict lexical functions do not have positive expansions into most bases in $QSym$ and $NSym$. 

\begin{example}\label{ex:cslf} Small expansions of (dual) column-strict lexical functions reveal negative coefficients:
\begin{align*}
&CL^*_{23}=\shin_{23}^*+\shin_{2111}^*-\shin_{1112}^*= \hat S_{23}+\hat S_{2111}-\hat S_{131}+\hat S_{122}+\hat S_{1211}\\
&CL^*_{22} = \mathfrak S^*_{22}-\mathfrak S^*_{13} \quad CL^*_{11}=\frac{1}{2}\Psi_{11}-\frac{1}{2}\Psi_{2}= \frac{1}{2}\Phi_{11}-\frac{1}{2}\Phi_{2},\\
 &\clx_{13} = H_{1,3}-H_{3,1} = R_{13}-R_{31}, \qquad
\clx_{11} = \frac{1}{2} \mathbf{\Psi}_{11} - \frac{1}{2}\mathbf{\Psi}_{2} = \frac{1}{2} \mathbf{\Phi}_{11} - \frac{1}{2}\mathbf{\Phi}_{2},\\
&\clx_{111} = \mathfrak{S}_{111} - \mathfrak{S}_{3} = \hat{\mathbf{s}}_{111} - \hat{\mathbf{s}}_{3}= \shin_{111} - \shin_{3}.
\end{align*}
    
\end{example}

For inquiries regarding remaining open questions, please contact the authors.

 \subsection*{Acknowledgements}
 The first author was supported through an NSERC Discovery Grant.
 
 \subsection*{AI Disclosure} The authors take full responsibility for the correctness and integrity of all mathematical results, proofs, and text. 
 Claude by Anthropic (Opus 4.8 and 5.5, Fable 5 and 5.2) contributed as follows:
(1) Reviewed the draft for typos and mathematical errors. (2) Identified the q-analogue results (Theorem \ref{thm:q-hook}, Proposition \ref{cor:stir}). (3) Identified a connection with maj-divisible tableaux and constructed the proof of Theorem \ref{thm:K-to-J}. (4) Computed the number of standard column-strict lexical tableaux and identified a correspondence with lattice paths (Proposition \ref{thm:rect_path}, Proposition \ref{prop: Motzkin}). (5) Generated Examples \ref{ex:tabs}, \ref{ex:lf}, and \ref{ex:cslf}. (6) Located a significant error in an earlier version of this work. (7) Computed expansions of lexical and dual lexical functions, leading to Conjecture \ref{conj:yqs}. (8) Suggested the comparison used in Remark \ref{rem:GR}. (9) Noted that $\prec$ preserves fundamental positivity. (10) Computed product expansions in Section \ref{sec:Prod}. 

ChatGPT-6 by OpenAI (Astra Pro) contributed as follows: (1) Reviewed the draft for typos and mathematical errors. (2) Suggested the application of Grinberg's strict-minimum restricted product using $N_m$ functions (Theorem \ref{thm:neck}) and framing against the dual immaculate basis. (3) Identified a secondary proof of $F$-positivity using the necklace characters. (4) Outlined the proof for Lemma \ref{lem:recursion} and Theorem \ref{thm:q-hook}, (5) Checked Conjecture \ref{conj:csF} computationally up to $n=8$. (6) Suggested changes to the framing of the proof of Proposition \ref{thm:rect_path}. (7) Proposed and proved Corollary \ref{cor:EtoL}. (8) Suggested Proposition \ref{prop:EtoL}. (9) Suggested alternate approach to Proposition \ref{cor:stir}.
\printbibliography

\end{document}